\documentclass{amsart}
\usepackage{amsmath,amssymb,amsthm,amsfonts,amscd,mathrsfs,mathtools}
\usepackage{tikz}
\usepackage[all]{xy}
\usepackage{enumerate}
\usetikzlibrary{calc}

\newcommand{\lemref}[1]{Lemma~\ref{#1}}

\theoremstyle{plain}
    \newtheorem{thm}{Theorem}[section]
    \newtheorem{lem}[thm]   {Lemma}
    
    \newtheorem{prop}[thm]  {Proposition}

\theoremstyle{definition}
    \newtheorem{defn}[thm]  {Definition}

    \newtheorem{ex}[thm]{Example}
    \newtheorem{rem}[thm]{Remark}

\def\ind{\mathrm{ind}}
\def\pind{\mathrm{p\,ind}}
\def\hind{\mathrm{h\,ind}}
\def\mind{\mathrm{m\,ind}}

\newcommand{\be}{\begin{enumerate}}
\newcommand{\ee}{\end{enumerate}}
\newcommand{\R}{\mathbb{R}}
\newcommand{\Z}{\mathbb{Z}}

\newcommand{\Q}{\mathbb{Q}}

\newcommand{\wt}{\widetilde}

\usepackage{color}

\begin{document}

\title[Poincar\'e--Hopf Theorem for line fields]{The Poincar\'e--Hopf Theorem for line fields revisited}

\author{Diarmuid Crowley}
\author{Mark Grant}

\address{Institute of Mathematics, University of Aberdeen, Fraser Noble Building, Meston Walk, Aberdeen AB24 3UE, UK}
\email{dcrowley@abdn.ac.uk}
\email{mark.grant@abdn.ac.uk}

\date{\today}


\begin{abstract}
A Poincar\'e--Hopf Theorem for line fields with point singularities
on orientable surfaces can be found Hopf's 1956 Lecture
Notes on Differential Geometry.  In 1955 Markus presented such a theorem in all dimensions, but Markus' statement only holds in even dimensions $2k \geq 4$.
In 1984 J\"{a}nich presented a Poincar\'{e}--Hopf theorem for line fields with more complicated singularities and focussed on the complexities arising in the generalised setting.

In this expository note we review the
Poincar\'e--Hopf Theorem for line fields with point singularities,
presenting a careful proof which is valid in all dimensions.


\end{abstract}

\maketitle
\section{Introduction}\label{sec:intro}

A line field is a smooth assignment of a tangent line at each point of a manifold, and may be thought of as a projective analogue of a vector field. More generally a line field may have a singular set where it is undefined. Line fields have come to prominence recently in soft matter physics where they are known also as \emph{nematic fields}, and their singularities as \emph{topological defects}.
In this setting they may be used to mathematically model certain types of ordered media. For example, nematic liquid crystals, which are materials formed of rod shaped molecules with no head or tail, can be reasonably modelled in this way \cite{Frank,KleLav}.
Much of the topological interest in line fields lies in the study and classification of their singularities, and for this the tools of homotopy theory have proven to be useful. We recommend Mermin's influential essay \cite{Mermin} or the colloquium \cite{Alexander} of Alexander \emph{et al} as readable introductions to these ideas.

The classical Poincar\'e--Hopf theorem \cite{HopfP,Milnor} relates the singularities of a vector field
to the Euler characteristic of the underlying manifold.
It states that for a vector field with finitely many isolated zeros on a compact manifold $M$, the sum of the indices at the zeros equals the Euler characteristic of $M$. There is an analogous but less well-known result for line fields with singularities, which is often quoted in the soft matter physics literature, and appears in the mathematical literature in various forms in works of Hopf, Markus, Koschorke and J\"anich (see the discussion below).

In this article we give a careful proof of the following Poincar\'e--Hopf theorem for line fields with singularities.

\begin{thm}\label{maintheorem}
Let $M^m$ be a compact manifold of dimension $m\ge2$, and let $\xi$ be a line field on $M$ with finitely many singularities $x_1,\ldots , x_n$. If $\partial M\neq\varnothing$, we assume additionally that the singularities lie in the interior of $M$, and that the line field is normal to $\partial M$. The projective index $\pind_\xi(x_i)$ of each singularity is defined (see Definition \ref{def:projindex}); it is an integer if $m$ is even, and an integer mod $2$ if $m$ is odd.

Let $\chi(M)$ denote the Euler characteristic of $M$. We have
\[
\sum_{i=1}^n \pind_\xi(x_i) = 2\chi(M),
\]
where the equality is interpreted as congruence mod $2$ when $m$ is odd.
\end{thm}

  Several statements similar to Theorem \ref{maintheorem} can be found in the mathematical literature, dating back to at least the 1950s. Perhaps the first such appears in the lecture notes of Heinz Hopf \cite[p.113]{Hopf} (where Poincar\'e is credited), and is stated only for orientable surfaces. Another, due to Lawrence Markus, appeared in an article in the Annals of Mathematics \cite[Theorem 2]{Markus}. Although it is stated for all dimensions, counter-examples may be given for surfaces and odd-dimensional manifolds (see Examples \ref{baseball} and \ref{RP3}).

  Our contribution is to give a unified proof of Theorem \ref{maintheorem} valid in all dimensions, thereby correcting the statement of \cite[Theorem 2]{Markus}, and generalising to higher dimensions the result in \cite{Hopf}. The proof we offer in Section \ref{sec:proof} is a correction of the proof in \cite{Markus}. One passes to a double branched cover associated to the line field which supports a vector field with isolated zeros, then applies the classical Poincar\'e--Hopf theorem and the Riemann--Hurwitz formula. The mistake made by Markus \cite{Markus} in the surface case (and rediscovered by the present authors) was in identifying the vector field indices in the double cover in terms of the projective indices in the original manifold. We introduce normal indices in Section 3 below in an attempt to clarify this rather subtle point.

In addition to the work of Hopf and Markus,
Koschorke \cite{Koschorke} and J\"anich \cite{Jaenich1,Jaenich2} have investigated line fields with singularities
in great detail. Koschorke's results \cite[Propositions 1.3 \& 1.8]{Koschorke} give a Poincar\'{e}-Hopf Theorem for line fields which implies Theorem \ref{maintheorem} above when $m\!>\!2$,
but Koschorke's definition of a singular line field is not the same as the one we use.
He considers a line field to be a vector bundle morphism $v:\xi\to TM$ from a line bundle $\xi$ on $M$ to the tangent bundle $TM$, and its singularities to be the points where $v$ drops rank. With this definition, every isolated singularity on a surface
is orientable; i.e.~has even projective index as defined in Definition \ref{def:projindex}.
Consequently the difficulties arising in the surface case when a line field cannot be extended over
a singularity do not arise in Koschorke's setting.

J\"anich \cite{Jaenich1,Jaenich2} investigates line fields with singularities from the viewpoint of obstruction theory (as suggested in \cite[Remark 1.9]{Koschorke}). His definition of a line field with singularities is more general than ours, in that he also considers the case where the singular set may have components of codimension two. Sections 1 and 2 of \cite{Jaenich2} contain a proof of Theorem \ref{maintheorem} along the lines of Koschorke \cite{Koschorke}, but treating $m=2$ as a special case. J\"anich shows that in the surface case, the sum of the projective indices may be viewed as the Poincar\'e dual of the cohomology class obstructing the existence of a line field without singularities \cite[Satz und Definition 1.3]{Jaenich2}. Hence this sum is independent of the particular line field. The value of the sum is then calculated by taking a line field which comes from a vector field \cite[2.3]{Jaenich2}.

It is our hope that this paper generates interest in questions of algebraic and differential topology arising in the theory of soft matter physics.

We thank Robert Bryant, Silke Henkes, Matthias Kreck and John Oprea for useful conversations and references to the literature. We especially thank Silke Henkes for acquainting us with the baseball line field (Example \ref{baseball}), and John Oprea for providing the construction given in Remark \ref{rem:Oprea}. We also thank the anonymous referee for helpful comments.

\section{Definitions and previous results}

Let $M^m$ be a smooth manifold of dimension $m\ge2$ and let $TM \to M$
be the tangent bundle of $M$.
A vector field on $M$ is a smooth
section $v \colon M \to TM$.
If a zero $x$ of $v$ is isolated one can define an integer $\ind_v(x_i)$, the \emph{index} of $v$ at
$x$; see Definition \ref{def:degree}.
Recall that the Euler characteristic of a compact manifold $M$ is defined to be the alternating sum of its Betti numbers:
$$ \chi(M) : = \sum_{i=0}^\infty (-1)^i{\rm rank}\,\bigl(H_i(M; \Q)) \bigr.$$

Let $v$ be a vector field on the compact manifold $M$ with finitely many zeros $\{x_1, \dots, x_n\} \subset M$. If $M$ has a boundary, then we require $v$ to be pointing outwards at all boundary points. The Poincar\'{e}-Hopf Theorem \cite{HopfP,Milnor} states that the Euler characteristic of $M$
agrees with the sum of the indices of $v$:
$$ \chi(M) = \sum_{j=1}^n \ind_v(x_j).$$
The following related statement is well-known, and is also called the Poincar\'{e}-Hopf Theorem by some authors.

\begin{prop}[{\cite[p.552]{AlexandroffHopf}}] \label{prop:PH2}
A closed manifold $M$ admits a non-vanishing vector field if and only if $\chi(M) = 0$.
\end{prop}

We next consider an analogue of Proposition \ref{prop:PH2} for line fields which we now define.
A {\em line field} $\xi$ on $M$ assigns to each
$x\in M$ a smoothly varying one-dimensional linear subspace
$\xi(x)\subset TM_x$ of the tangent space to $M$ at $x$.
To define a line field as a section, we
let $PTM \to M$ be the projectivization of the tangent bundle of $M$;
i.e.~$PTM$ is the quotient space $TM/\simeq$ where $w_1 \simeq w_2$ if $w_1 = \lambda w_2$
for $\lambda \in \R \setminus \{0\}$.

\begin{defn}\label{def:linefield}
A \emph{line field} on $M$ is a smooth section $\xi: M \to PTM$.
\end{defn}

We may also view a line field as a one-dimensional sub-bundle $\xi\subseteq TM$ of the tangent bundle of $M$. When $M$ is endowed with a Riemannian metric, the unit sphere bundle of $\xi$ determines an \emph{associated double cover} $$\pi:\widetilde{M}\coloneqq S\xi\to M.$$
%
Clearly a non-vanishing vector field $v$ gives rise to a line field $\xi$ by taking
$$\xi(x) := \langle v(x)\rangle \subseteq TM_x$$
to be the line spanned by $v(x)$ at each point. Not every line field arises from a vector field in this way, however. In fact, it is easily seen that a line field $\xi$ lifts to a vector field if and only if its associated double cover $\pi:\widetilde{M}\to M$ is trivial.

Despite the above observation, it turns out that
the existence of a line field on $M$ is equivalent to the existence of a non-vanishing vector field on $M$. The proof we give below (due to Markus \cite{Markus}) is non-constructive, but is a useful warm-up for the proof of Theorem \ref{maintheorem}.

Note that a line field $\xi$ on $M$ defines a canonical vector field $\wt \xi$ on the total space of the double cover $\wt M$.  To define the vector field $\wt \xi$, we write points in $\wt M$ as pairs $(x, w)$ where $w \in \xi(x)$ has unit norm, and we identify $T\wt M_{(x, w)} = TM_x$.  Then we define
\[ \wt \xi(x, w) = w \in T\wt M_{(x, w)} = TM_x.   \]

\begin{thm}\label{lineiffvec}
A manifold $M$ admits a line field if and only if it admits a non-vanishing vector field.
\end{thm}

\begin{proof}
It is well-known that any manifold which is non-compact or has non-empty boundary admits a non-vanishing vector field, and hence a line field. We may therefore assume $M$ to be closed.

Let $\xi$ be a line field on $M$, let $\pi:\widetilde{M}\to M$ be the associated double cover, 
and let $\wt\xi$ be the canonical vector field on $\wt M$ associated to $\xi$.
Since $\wt\xi$ is non-vanishing,
by the classical Poincar\'e--Hopf Theorem and the multiplicativity of the Euler characteristic in coverings, we have
\[
0 = \chi(\widetilde{M}) = 2 \, \chi(M).
\]
Therefore $\chi(M)=0$, and $M$ admits a non-vanishing vector field by Proposition \ref{prop:PH2}.
\end{proof}

\begin{rem}\label{rem:Oprea}
We offer the following more constructive proof of Theorem \ref{lineiffvec}. Suppose $M$ admits a line field, which we view as a line sub-bundle $\xi\subseteq TM$. Take a section $s:M\to \xi$ which is transverse to the zero section of $\xi$. Assuming $M$ to be compact, the zeroes of $s$ form a finite set $\{x_1,\ldots , x_k\}\subseteq M$. The tangent bundle of $M$ decomposes as $TM\cong \xi\oplus E$ for some $(m-1)$-dimensional bundle $E$. We may modify the zero section of $E$ by adding a bump function on a small neighbourhood of each zero $x_i$ of $s$ (on which the bundle $E$ is trivial), thus obtaining a section $t:M\to E$ which is zero outside of these neighbourhoods. Then setting $v(x) = (s(x),t(x))\in \xi\oplus E\cong TM$ defines a non-vanishing vector field on $M$.
\end{rem}

\begin{rem}
For any given line bundle $\xi$ on $M$, one can ask the more refined question,
``Is there an embedding $\xi \subset TM$?"  This question is answered completely
in \cite[Theorems 2.1]{Koschorke}; see also \cite[Theorem 3.1]{Koschorke}.
\end{rem}

Our main interest is in line fields with finitely many isolated singularities.

\begin{defn}\label{def:linefieldsing}
A \emph{line field on $M$ with singularities} at $x_1,\ldots , x_n\in M$ is a line field on the complement $M\setminus\{x_1,\ldots , x_n\}$.
\end{defn}

For instance, a vector field $v$ on a closed manifold $M$ with isolated zeros at
$x_1,\ldots, x_n$ determines such a line field with singularities (but the converse does not hold). 
We recall that for such a vector field the classical Poincar\'e--Hopf theorem asserts that
\[
\sum_{i=1}^n \ind_v(x_i) = \chi(M).
\]
As mentioned in the Introduction, we are aware of several analogous results for line fields with singularities in the literature, two of which we will now discuss.

The first is due to H.\ Hopf and appears in \cite[pp.107--113]{Hopf}. Let $x$ be an isolated singularity of a line field $\xi$ on a surface $\Sigma$. Hopf associates to such a singularity a half-integer index, which we denote $\hind_\xi(x)\in \frac12\Z$ and call the \emph{Hopf index}, defined as follows. Take a small closed curve $C:[0,1]\to \Sigma$ around $x$, with no other singularities on $C$ or in its interior. Choose a vector lifting of $\xi(C(0))$, which by continuity determines a vector lifting of $\xi(C(t))$ for all $0\le t\le 1$. If $C$ was chosen small enough to be contained in some coordinate chart, we can measure the number of total rotations of this vector as $C$ is traversed, relative to the local coordinates. Hopf shows that this gives a half-integer $\hind_\xi(x)\in\frac12\Z$ which is independent of the various choices involved.

Sketches of line field singularities of various indices appearing in \cite{Hopf} are reproduced in Figure \ref{fig:Hopfindices} (similar figures appeared at around the same time in the physics literature \cite{Frank}). In each case it is the integral curves of the line field which are shown.

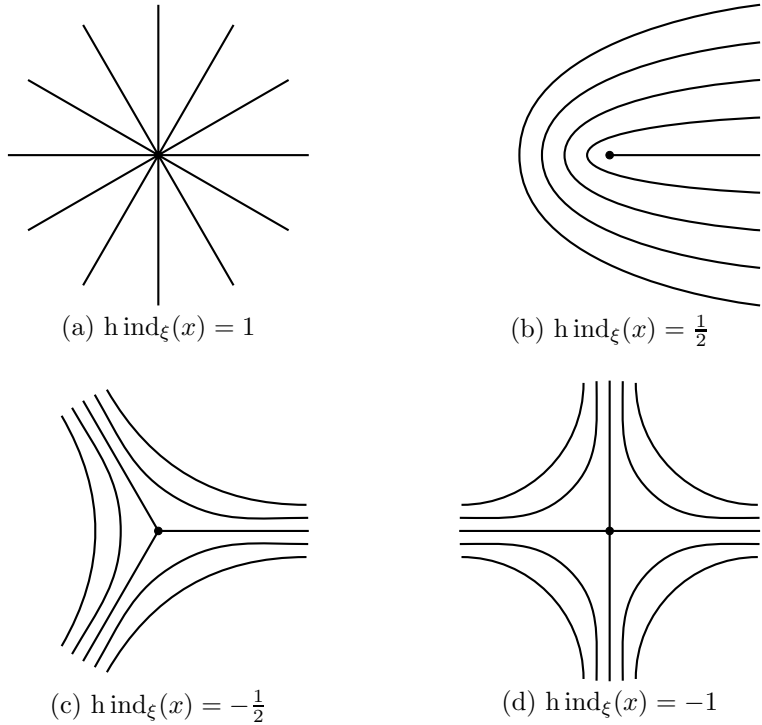
\begin{figure}
\begin{tikzpicture}
\draw[fill] (-3,0) circle[radius=0.05];
\foreach \i in {0,...,11}
\draw[thick](-3,0)--($(-3,0)+({30*\i}:2cm)$);
\node[anchor=north] at(-3,-2){(a) $\hind_\xi(x)=1$};

\draw[fill] (3,0) circle[radius=0.05];
\draw[thick] (3,0)--(5,0);
\foreach \i in {1,...,4}
\draw[thick] plot [smooth,tension=1.5] coordinates {($(5,0)+0.5*(0,\i)$) ($(3,0)-0.3*(\i,0)$)($(5,0)-0.5*(0,\i)$)};
\node[anchor=north] at(3,-2){(b) $\hind_\xi(x)=\frac12$};

\draw[fill] (-3,-5) circle[radius=0.05];
\foreach \i in {1,2,3}
\draw[thick](-3,-5)--($(-3,-5)+(120*\i:2cm)$);
\foreach \i in {1,2,3}
\draw[thick] ($(-3,-5)+(120*\i+5:2cm)$) to [out={120*\i+180},in={120*\i-30}] ($(-3,-5)+(120*\i+60:0.5cm)$) to [out={120*\i+150},in={120*\i-60}] ($(-3,-5)+(120*\i+115:2cm)$);
\foreach \i in {1,2,3}
\draw[thick] ($(-3,-5)+(120*\i+10:2cm)$) to [out={120*\i+180},in={120*\i-60}] ($(-3,-5)+(120*\i+110:2cm)$);
\node[anchor=north] at(-3,-7){(c) $\hind_\xi(x)=-\frac12$};

\draw[fill] (3,-5) circle[radius=0.05];
\foreach \i in {1,...,4}
\draw[thick](3,-5)--($(3,-5)+(90*\i:2cm)$);
\foreach \i in {1,...,4}
\draw[thick] ($(3,-5)+(90*\i+5:2cm)$) to [out={90*\i+180},in={90*\i-45}] ($(3,-5)+(90*\i+45:0.75cm)$) to [out={90*\i+135},in={90*\i-90}] ($(3,-5)+(90*\i+85:2cm)$);
\foreach \i in {1,...,4}
\draw[thick] ($(3,-5)+(90*\i+10:2cm)$) to [out={90*\i+180},in={90*\i-90}] ($(3,-5)+(90*\i+80:2cm)$);
\node[anchor=north] at(3,-7){(d) $\hind_\xi(x)=-1$};

\end{tikzpicture}
\caption{Line field singularities and their Hopf indices}
\label{fig:Hopfindices}
\end{figure}

\begin{thm}[{\cite[p.113]{Hopf}}] For a line field $\xi$ with singularities $x_1,\ldots , x_n$ on a closed orientable surface $\Sigma$,
\[
\sum_{i=1}^n \hind_\xi(x_i) = \chi(\Sigma).
\]
\end{thm}

\begin{rem}
The assumption that $\Sigma$ is orientable is easily removed (for example, by passing to an oriented cover). The argument in \cite{Hopf} proceeds by showing that $2\pi$ times the sum of the indices is the integral of the Gaussian curvature (this gives a proof of the Gauss--Bonnet theorem). It is possible that this argument can be generalized to higher dimensions using the Gauss--Bonnet theorem for Riemannian manifolds due to Allendoerfer--Weil \cite{AW} and Chern \cite{Chern}.
\end{rem}

The second such result is due to L.\ Markus
\cite{Markus}.
To state it requires some preparation. Let $x$ be an isolated singularity of a line field $\xi$ on a closed manifold $M^m$ of arbitrary dimension $m\ge2$. We call the singularity $x$ \emph{orientable} if $\xi$ lifts to a vector field in a neighbourhood of $x$, or equivalently, if the restriction $\pi|_S:\widetilde{S}\to S$ of the associated double cover to a small sphere $S$ centred at $x$ and not containing any other singularities is trivial. Otherwise, we say that $x$ is \emph{non-orientable}. Note that when $m>2$ all singularities are orientable, and when $m=2$ the orientable singularities are those for which the Hopf index $\hind_\xi(x_i)$ is an integer.

Markus defines an integer projective index $\mind_\xi(x)\in\Z$, as follows. Restricting $\xi$ to $S$ produces a section $\xi|_S:S\to PTM|_S$. If $S$ was chosen small enough, we can compose with a trivialisation $PTM|_S\to S\times P^{m-1}$ and then project onto the second coordinate to obtain a map $$f:S^{m-1}\approx S\to P^{m-1}.$$ If $m$ is even, we can orient $S$ and $P^{m-1}$ and define $\mind_\xi(x)$ to be the degree of this map. If $m\ge3$ is odd, then $x$ is orientable and it follows that $f$ lifts through the standard double cover $S^{m-1}\to P^{m-1}$ to a map $\widetilde{f}:S\to S^{m-1}$. Choosing base points determines an element $[\widetilde{f}]\in \pi_{m-1}(S^{m-1})$, which on suspending gives an element in $\pi_m(S^m)$. Composing a representative of this element with the double cover $S^m\to P^m$ gives a map $g:S^m\to P^m$, whose degree is taken to be $\mind_\xi(x)$.
We remark that $\mind_\xi(x)$ is always an {\em even} integer; this is because the double cover $S^m \to P^m$ has
degree two.

 Markus then claims \cite[Theorem 2]{Markus} that if $\xi$ is a line field on $M$ with singularities at $x_1,\ldots , x_n$, precisely $k$ of which are non-orientable, then
\[
\sum_{i=1}^n \mind_\xi(x_i) = 2\, \chi(M) - k.
\]
Unfortunately, there are counter-examples to this result when $m=2$ or $m$ is odd.

\begin{ex}\label{baseball}
This example is known colloquially among soft matter physicists as the ``baseball". It is a line field on $S^2$ with four singularities of Hopf index $\frac12$ (and Markus index $1$). Figure \ref{fig:baseball} illustrates how to produce such a line field by glueing together two copies of a line field on the disk $D^2$ which is parallel to the boundary.

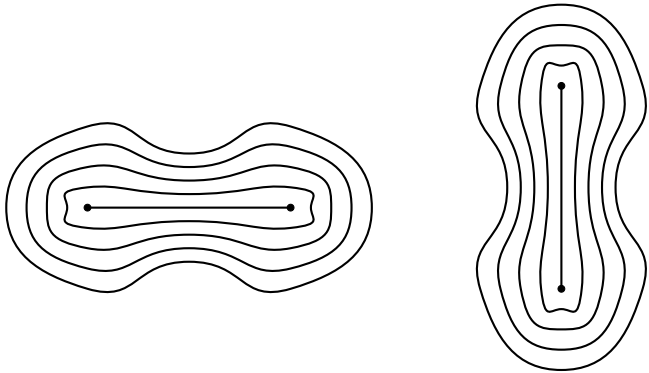
\begin{figure}
\begin{tikzpicture}
\newcommand\halfbaseball{
\draw[fill] (-3,0) circle[radius=0.05];
\draw[fill] (0,0) circle[radius=0.05];
\draw[thick] (-3,0)--(0,0);
\foreach \i in {1,...,4}
\draw[thick] plot [smooth cycle,tension=1] coordinates {(0.3*\i,0) (0,0.3*\i) ($(-1.5,0)+0.2*(0,\i)$) ($(-3,0)+0.3*(0,\i)$) ($(-3,0)-0.3*(\i,0)$)($(-3,0)-0.3*(0,\i)$)($(-1.5,0)-0.2*(0,\i)$) (0,-0.3*\i)};
}
\begin{scope}[scale=0.9]
\halfbaseball
\end{scope}
\begin{scope}[scale=0.9,xshift=4cm,yshift=1.8cm,rotate=90]
\halfbaseball
\end{scope}
\end{tikzpicture}
\caption{The baseball line field on $S^2$ is obtained by glueing these two disks along their boundaries.}
\label{fig:baseball}
\end{figure}

In this example, all four singularities are non-orientable, and the sum of their Markus indices is $4$. This contradicts Markus' result, since
\[
4 \neq 2\,\chi(S^2) - 4  = 0.
\]
\end{ex}

\begin{ex}\label{RP3}
We may construct a line field on real projective space $P^m$ with a single singularity of Markus index $2$. Consider the line field on the standard disk $D^m\subset\R^m$ with a single singularity, whose integral curves are the rays in $\R^m$ emanating from the origin (the case $m=2$ is as shown in Figure \ref{fig:Hopfindices}(a)). Regarding $P^m$ as obtained from $D^m$ by identifying antipodal boundary points, we obtain a line field on $P^m$ with a single singularity of Markus index $2$. When $m\ge3$ is odd this gives a counter-example to Markus' result, since then $\chi(P^m)=0$.
\end{ex}

Theorem \ref{maintheorem} in the introduction fixes the statement and proof given by Markus \cite{Markus}, thereby extending the Theorem of Hopf \cite{Hopf} to all dimensions.

\section{Indices and projective indices}

In this section we recall the classical definition of the index of an isolated zero of a vector field, and give our definition of the projective index of an isolated singularity of a line field. We then define alternative indices, called \emph{normal indices}, which are more suited to proving Theorem \ref{maintheorem}, and describe how they relate to the classical indices.

Let $M$ be a closed smooth manifold of dimension $m\ge2$, and let $v:M\to TM$ be a vector field on $M$ with an isolated zero at $x\in M$. Let $D$ be a small coordinate disk centred at $x$ not containing any other zeros of $v$, and let $S\approx S^{m-1}$ be its boundary. Restricting $v$ to $S$ and normalizing results in an embedding $v:S\to STM|_S$. On the other hand, a trivialisation of the sphere tangent bundle over $S$ gives a diffeomorphism $\Phi: STM|_S\to S\times S^{m-1}$.
\begin{defn}[Cf.~{\cite[\S 6]{Milnor}}]
\label{def:degree}
The index of $v$ at $x$, denoted $\ind_v(x) \in \Z$, is the Brouwer degree of the composition
\[
\xymatrix{
f: S \ar[r]^-{v} & STM|_S \ar[r]^-{\Phi} & S\times S^{m-1} \ar[r]^-{\pi_2} & S^{m-1},
}
\]
where $\pi_2$ denotes projection onto the second factor.
\end{defn}

For our later arguments, it is important to identify the index of $v$ at $x$
as the intersection number of certain sections of $TM$ restricted to $S$.
Observe that any point $a\in S^{m-1}$ determines a section $\sigma=\sigma_a:S\to STM|_S$ defined by $\sigma(z)=\Phi^{-1}(z,a)$. For generic $a$ the embeddings $\sigma$ and $v$ will intersect transversely. Choosing a local orientation at $x$ results in an orientation of the base sphere $S$ and the fibre sphere $S^{m-1}$, and hence a product orientation of $S\times S^{m-1}$ and therefore (using $\Phi$) an orientation of $STM|_S$. It follows that there is a
well-defined intersection number
$$\sigma(S)\pitchfork v(S)\in \Z.$$
%
Using the definition of degree from differential topology (in terms of the signs of the differential at pre-images of a regular value $a\in S^{m-1}$), it is not difficult to check
that this intersection number gives the index of $v$ at $x$.  Hence we have

\begin{lem}\label{lem:index}
The index of $v$ at $x$ satisfies $\ind_v(x) = \sigma(S)\pitchfork v(S)$.
 \end{lem}

Fixing a Riemannian metric on $M$ determines an outward unit normal of the codimension one embedding $S\hookrightarrow M$, which we regard as an embedding $\eta: S\to STM|_S$.

 \begin{defn}\label{def:perpindex}
 The \emph{normal index} of the vector field $v$ at the zero $x$, denoted $\ind^\perp_v(x)$, is defined to be the oriented intersection number $\eta(S)\pitchfork v(S)\in \Z$, where $\eta$ is the outward unit normal to the sphere $S\hookrightarrow M$.
 \end{defn}

\begin{lem}\label{lem:indperpind}
We have
\[
\ind^\perp_v(x) = \ind_v(x) + (-1)^{m-1}.
\]
\end{lem}

\begin{proof}
We proceed homologically, calculating intersection numbers in the product $S\times S^{m-1}$ (compare Bredon \cite[VI, Example 11.12]{Bredon}). Let $[S]\in H_{m-1}(S)$ and $[S^{m-1}]\in H_{m-1}(S^{m-1})$ be the fundamental classes, so that $[S]\times [S^{m-1}]\in H_{2m-2}(S\times S^{m-1})$ is the fundamental class of the product. Denote the (Kronecker) dual cohomology classes by $\alpha\in H^{m-1}(S)$ and $\beta\in H^{m-1}(S^{m-1})$. In the cohomology of the product
\[
H^*(S\times S^{m-1}) \cong H^*(S)\otimes H^*(S^{m-1}),
\]
the Poincar\'e dual of the homology class $\Phi_*\sigma_*[S]=[S]\times 1 \in H_{m-1}(S\times S^{m-1})$ is the class $(-1)^{m-1}\times\beta$, and the Poincar\'e dual of $\Phi_*\eta_*[S]=1\times[S^{m-1}]+[S]\times 1$ is the class $(\alpha\times 1 + (-1)^{m-1} \times \beta)$. Let $\gamma\in H^{m-1}(S\times S^{m-1})$ be the Poincar\'e dual of $\Phi_*v_*[S]$.
Applying
Lemma \ref{lem:index} and Definition \ref{def:perpindex}
we therefore find that

\begin{align*}
\ind_v^\perp(x) & = \eta(S)\pitchfork v(S) \\
                & = \langle(\alpha\times 1 + (-1)^{m-1} \times \beta)\cup\gamma,[S]\times [S^{m-1}]\rangle\\
                & = \langle(\alpha\times 1)\cup\gamma,[S]\times[S^{m-1}]\rangle +  \langle((-1)^{m-1}\times\beta)\cup\gamma, [S]\times [S^{m-1}]\rangle\\
                & = \langle(\alpha\times 1)\cup\gamma,[S]\times[S^{m-1}]\rangle + \sigma(S)\pitchfork v(S)\\
                & = (-1)^{m-1} + \ind_v(x) ,
\end{align*}
as was to be shown. (The first term is evaluated by noticing that it agrees with the oriented intersection number of the fibre with the base, in that order.)
\end{proof}

We now define the projective index of an isolated singularity of a line field.
Let $\xi$ be a line field on $M$ with finitely many isolated singularities, that is, a section $\xi:M\setminus\{x_1,\ldots , x_n\}\to PTM|_{M\setminus\{x_1,\ldots , x_n\}}$ of the restriction of the projectivized tangent bundle of $M$ to the complement of a finite set of points. Let $x\in M$ be a singularity of $\xi$. Restricting to a small sphere $S$ containing $x$ as before gives a section $\xi:S\to PTM|_S$. As before we have a trivialization $\Phi:PTM|_S\to S\times P^{m-1}$ and therefore any point $a\in P^{m-1}$ determines a section $\sigma=\sigma_a:S\to PTM|_S$ defined by $\sigma(z)=\Phi^{-1}(z,a)$. When $m$ is even, a local orientation at $x$ determines orientations of $S$, the fibre $P^{m-1}$, the product $S\times P^{m-1}$ and $PTM|_S$.

\begin{defn}\label{def:projindex}
The \emph{projective index} of the line field $\xi$ at the singularity $x$, denoted $\pind_\xi(x)$, is defined to be the oriented intersection number $\sigma(S)\pitchfork \xi(S)\in \Z$ when $m$ is even, and the mod $2$ intersection number $\sigma(S)\pitchfork_2 \xi(S)\in \Z/2$ when $m$ is odd.
 \end{defn}

\begin{rem}
The projective index defined above
agrees with Markus' index \cite{Markus} in the case $m$ even, and is twice Hopf's index \cite[p. 107]{Hopf} in the case $m=2$.
The Markus index for $m\ge3$ odd seems not to be well-defined, due to there being two possible choices for the lift $\wt{f}:S\to S^{m-1}$ which differ by a map of degree $(-1)^{m}=-1$.
In the odd-dimensional case it is possible to define a projective index taking values in the non-negative integers (compare Mermin \cite[p.\,631]{Mermin} or Alexander \emph{et al} \cite[p.\,506]{Alexander}), but this will not be used here.
\end{rem}

Fixing a Riemannian metric on $M$ determines a normal line to the codimension one embedding $S\hookrightarrow M$, which we regard as an embedding $\eta: S\to PTM|_S$.

\begin{defn}\label{def:perpprojindex}
 The \emph{normal projective index} of the line field $\xi$ at the singularity $x$, denoted $\pind^\perp_\xi(x)$, is defined to be the oriented intersection number $\eta(S)\pitchfork \xi(S)\in \Z$ when $m$ is even, and the mod $2$ intersection number $\eta(S)\pitchfork_2 \xi(S)$ when $m$ is odd.
 \end{defn}

 \begin{lem}\label{lem:projindperpprojind}
 When $m$ is even, we have
\[
\pind^\perp_\xi(x) = \pind_\xi(x) -2.
\]
\end{lem}

\begin{proof}
As above we proceed homologically, computing intersection numbers in the product $S\times P^{m-1}$. Let $[S]\in H_{m-1}(S)$ and $[P^{m-1}]\in H_{m-1}(P^{m-1})$ denote the fundamental classes, and let $\alpha\in H^{m-1}(S)$ and $\beta\in H^{m-1}(P^{m-1})$ denote their Kronecker duals. In the cohomology of the product
\[
H^*(S\times P^{m-1}) \cong H^*(S)\otimes H^*(P^{m-1}),
\]
the Poincar\'e dual of the homology class $\Phi_*\sigma_*[S]=[S]\times 1 \in H_{m-1}(S\times P^{m-1})$ is the class $-1\times\beta$, and the Poincar\'e dual of $\Phi_*\eta_*[S]=2\times[P^{m-1}]+[S]\times 1$ is the class $(\alpha\times 2 - 1 \times \beta)$. Let $\gamma\in H^{m-1}(S\times P^{m-1})$ be the Poincar\'e dual of $\Phi_*\xi_*[S]$. Using Definitions \ref{def:projindex} and \ref{def:perpprojindex} we therefore have

\begin{align*}
\pind^\perp_\xi(x) & = \eta(S)\pitchfork \xi(S) \\
                   & = \langle (\alpha\times 2 - 1 \times \beta)\cup\gamma, [S]\times [P^{m-1}]\rangle \\
                   & = \langle (\alpha\times 2)\cup\gamma, [S]\times [P^{m-1}]\rangle + \langle (-1 \times \beta)\cup\gamma, [S]\times [P^{m-1}]\rangle \\
                   & = 2\langle (\alpha\times 1)\cup\gamma, [S]\times [P^{m-1}]\rangle + \sigma(S)\pitchfork \xi(S) \\
                   & = -2 + \pind_\xi(x),
\end{align*}
as required. (Here the first term is twice the intersection number of the fibre and the base, which is $-2$ if $m$ is even.)
\end{proof}

\begin{lem}\label{lem:oddpindzero}
When $m\ge3$ is odd, we have
\[
\pind_\xi(x)= \pind_\xi^\perp(x)=0\in\Z/2.
\]
\end{lem}

\begin{proof}
Similarly to Definition \ref{def:degree} and Lemma \ref{lem:index}, we may regard the projective index as the mod $2$ degree of the composition
\[
\xymatrix{
 f: S \ar[r]^-{\xi} & PTM|_S \ar[r]^-{\Phi} & S\times P^{m-1} \ar[r]^-{\pi_2} & P^{m-1}.
 }
 \]
 However, using covering space theory (and since $m\ge 3$) it is easy to see that this map $f:S\to P^{m-1}$ factors through the canonical double cover $S^{m-1}\to P^{m-1}$, and therefore its mod $2$ degree is zero.

 Equality of the projective index and the normal projective index is clear, since the sections $\sigma$ and $\eta$ represent the same mod $2$ homology class.
\end{proof}

\section{Proof of Theorem \ref{maintheorem}}\label{sec:proof}

We will assume for the proof that $M^m$ is closed. The case when $\partial M \neq \varnothing$ then follows by an obvious argument involving the double of $M$.

 Let $\xi$ be a line field on $M$ with isolated singularities at $x_1,\ldots , x_n$. For each $i=1,\ldots , n$, let $D_i$ be a coordinate disk centred at $x_i$ which does not contain any other singularities and let $S_i=\partial D_i$ be its boundary. Removing the interiors of all the $D_i$ results in a compact manifold $N$ with boundary
\[
\partial N = \bigsqcup_{i=1}^n S_i \approx \bigsqcup_{i=1}^n S^{m-1}.
\]
 The restriction $\xi|_N$ is a global line field on $N$. Let $p:\widetilde{N}\to N$ denote the double cover associated to $\xi|_N$. The restriction of $p$ to each boundary component is a double cover $p_i:\widetilde{S}_i\coloneqq p^{-1}(S_i)\to S_i$ of an $(m-1)$-sphere, which is trivial if and only if $x_i$ is orientable.

 By glueing in disks along the boundary components of $\widetilde{N}$, we obtain a closed manifold $\widetilde{M}$ with a double cover $\pi:\widetilde{M}\to M$ extending $p$. This covering may be branched in the case $m=2$, with branch points of ramification index $2$ above the non-orientable singularities.
 The restriction $\xi|_N$ of the line field to $N$ lifts canonically to a vector field $\widetilde{\xi|_N}$ on $\widetilde{N}$,
 and this extends (by scaling radially on each disk) to a vector field $v$ on $\widetilde{M}$ with isolated zeroes. At each singularity $x_i$ of $\xi$, the pre-image $\pi^{-1}(x_i)$ consists of one or two zeroes of $v$, according as $x_i$ is a non-orientable or an orientable singularity.

 We are going to apply the classical Poincar\'e--Hopf theorem to the vector field $v$ on $\widetilde{M}$ and the Riemann--Hurwitz formula to deduce our main Theorem. For this we need to relate the projective indices at the singularities on $M$ to the vector field indices at the covering zeroes.

 \begin{lem}\label{upstairsdownstairs}
 When $m$ is even, for each singularity $x$ of the line field $\xi$ we have
 \[
 \pind^\perp_\xi(x) = \sum_{y\in \pi^{-1}(x)} \ind^\perp_v(y).
 \]
 \end{lem}

 \begin{proof}

 This is intuitively clear: the number of times the line field is normal to the sphere $S$ equals the number of times the associated vector field on the double cover agrees with the outward normal on $\widetilde{S}$. See Figure \ref{fig:cover}.

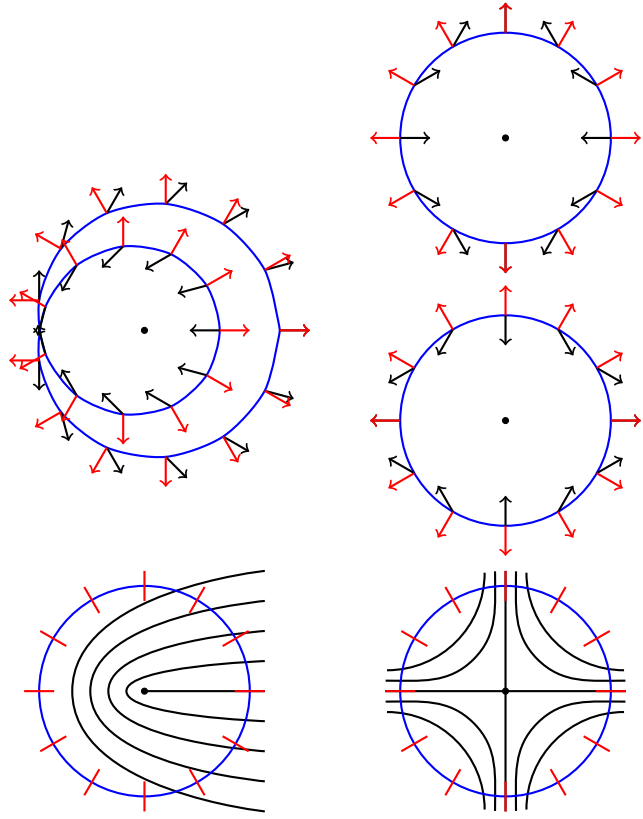
\begin{figure}
\begin{tikzpicture}[scale=0.8]

\begin{scope}[xshift=-6cm,yshift=-5cm]
\draw[fill] (3,0) circle[radius=0.05];
\draw[thick] (3,0)--(5,0);
\foreach \i in {1,...,4}
\draw[thick] plot [smooth,tension=1.5] coordinates {($(5,0)+0.5*(0,\i)$) ($(3,0)-0.3*(\i,0)$)($(5,0)-0.5*(0,\i)$)};
\end{scope}

\begin{scope}
\draw[fill] (3,-5) circle[radius=0.05];
\foreach \i in {1,...,4}
\draw[thick](3,-5)--($(3,-5)+(90*\i:2cm)$);
\foreach \i in {1,...,4}
\draw[thick] ($(3,-5)+(90*\i+5:2cm)$) to [out={90*\i+180},in={90*\i-45}] ($(3,-5)+(90*\i+45:0.75cm)$) to [out={90*\i+135},in={90*\i-90}] ($(3,-5)+(90*\i+85:2cm)$);
\foreach \i in {1,...,4}
\draw[thick] ($(3,-5)+(90*\i+10:2cm)$) to [out={90*\i+180},in={90*\i-90}] ($(3,-5)+(90*\i+80:2cm)$);
\end{scope}

\draw[thick,blue] (-3,-5) circle[radius=1.75];
\foreach \j in {0,...,11}
\draw[thick,red] ($(-3,-5)+(30*\j:1.5cm)$)--($(-3,-5)+(30*\j:2cm)$);

\draw[thick,blue] (3,-5) circle[radius=1.75];
\foreach \j in {0,...,11}
\draw[thick,red] ($(3,-5)+(30*\j:1.5cm)$)--($(3,-5)+(30*\j:2cm)$);

\begin{scope}[yshift=1cm]
\draw[fill] (-3,0) circle[radius=0.05];
\def\pts{}
\foreach \x in {0,...,24}{\xdef\pts{\pts ($(-3,0)+(30*\x:1.75cm)+(15*\x:.5cm)$)};}
\draw[thick,blue] plot [smooth] coordinates {\pts};
\foreach \y in {0,...,23}
\draw[->,thick]($(-3,0)+(30*\y:1.75cm)+(15*\y:.5cm)$)--($(-3,0)+(30*\y:1.75cm)+(15*\y:1cm)$);
\foreach \y in {0,...,23}
\draw[->,thick,red]($(-3,0)+(30*\y:1.75cm)+(15*\y:.5cm)$)--($(-3,0)+(30*\y:1.75cm)+(15*\y:.5cm)+(30*\y:.5cm)$);
\end{scope}

\draw[thick,blue] (3,-.5) circle[radius=1.75];
\draw[fill] (3,-.5) circle[radius=0.05];
\foreach \k in {0,...,11}
\draw[thick,->] ($(3,-.5)+(30*\k:1.75cm)$)--($(3,-.5)+(30*\k:1.75cm)+(-30*\k:.5cm)$);
\foreach \j in {0,...,11}
\draw[thick,red,->] ($(3,-.5)+(30*\j:1.75cm)$)--($(3,-.5)+(30*\j:2.25cm)$);

\begin{scope}[yshift=.2cm]
\draw[fill] (3,4) circle[radius=0.05];
\draw[thick,blue] (3,4) circle[radius=1.75];
\foreach \k in {0,...,11}
\draw[thick,->] ($(3,4)+(30*\k:1.75cm)$)--($(3,4)+(30*\k:1.75cm)-(-30*\k:.5cm)$);
\foreach \j in {0,...,11}
\draw[thick,red,->] ($(3,4)+(30*\j:1.75cm)$)--($(3,4)+(30*\j:2.25cm)$);
\end{scope}

\end{tikzpicture}
\caption{An illustration of Lemma \ref{upstairsdownstairs} in the case $m=2$. On the bottom left is a non-orientable singularity with $\pind_\xi(x)=1$ and $\pind^\perp_\xi(x)=-1$. On the bottom right is an orientable singularity with $\pind_\xi(x)=-2$ and $\pind^\perp_\xi(x)=-4$.  Above these are the lifted vector fields on the double cover, restricted to a small bounding circle (shown in blue). Normal fields are shown in red.}
\label{fig:cover}
\end{figure}

For a smooth map $f:A^a\to B^b$ between closed oriented manifolds, denote by $f_!:H^*(A)\to H^{*+b-a}(B)$ the pushforward map (obtained from the induced map on homology by pre- and post-composing with Poincar\'e duality isomorphisms). With this notation, the Poincar\'e dual of $f_*[A]$ is $f_!(1)$.

Give $\widetilde{M}$ the Riemannian metric induced by the double cover $\pi:\widetilde{M}\to M$. This covering induces a $4$-fold covering $\overline{\pi}:ST\widetilde{M}|_{\widetilde{S}}\to PTM|_{S}$. Taking the pullback of this cover along the normal line $\eta:S \to PTM|_S$ results in a pullback square
\[
\xymatrix{
\widetilde{S}\sqcup\widetilde{S} \ar[r]^-{\widetilde{\eta}\sqcup -\widetilde{\eta}} \ar[d]_{p\sqcup p} & ST\widetilde{M}|_{\widetilde{S}} \ar[d]_{\overline{\pi}} \\
S \ar[r] ^-{\eta} & PTM|_S
}
\]
where $\widetilde{\eta}:\widetilde{S}\to ST\widetilde{M}|_{\widetilde{S}}$ denotes the outward unit normal to $\widetilde{S}$.

Applying the push-pull formula to this square, we obtain
\begin{align*}
\overline{\pi}^* \eta_!(1) & = \left(\widetilde{\eta}\sqcup -\widetilde{\eta}\right)_!\left(p\sqcup p\right)^*(1) \\
                         & = \widetilde{\eta}_!(1) + (-\widetilde{\eta})_!(1) \\
                         & = 2\,\widetilde{\eta}_!(1).
\end{align*}
Here we have used the fact that the inward pointing normal $-\widetilde{\eta}:\widetilde{S}\to ST\widetilde{M}|_{\widetilde{S}}$ is obtained from $\widetilde{\eta}$ by composing with a map of degree $(-1)^m=1$ on the fibres, hence they represent the same cohomology class.

A similar argument shows that $\overline{\pi}^*\xi_!(1) = 2\, v_!(1)$. Therefore we compute:
\begin{align*}
4\, \pind^\perp_\xi(x) & = 4\, \langle \eta_!(1)\cup \xi_!(1), [PTM|_S]\rangle \\
                     & = \langle \eta_!(1)\cup \xi_!(1), 4[PTM|_S]\rangle \\
                     & = \langle \eta_!(1)\cup \xi_!(1), \overline{\pi}_*[ST\widetilde{M}|_{\widetilde{S}}]\rangle \\
                     & = \langle \overline{\pi}^*\eta_!(1)\cup \overline{\pi}^*\xi_!(1), [ST\widetilde{M}|_{\widetilde{S}}]\rangle \\
                     & = \langle 2\widetilde{\eta}_!(1)\cup 2 v_!(1), [ST\widetilde{M}|_{\widetilde{S}}]\rangle \\
                     & = 4 \sum_{y\in \pi^{-1}(x)} \ind^\perp_v(y),
\end{align*}
which proves the Lemma.
 \end{proof}

  We can now complete the proof of Theorem \ref{maintheorem} in the case $m$ even. Let $k$ be the number of non-orientable singularities. By the Riemann--Hurwitz formula, $\chi(\widetilde{M})=2\chi(M)-k$. Note that the vector field $v$ on $\widetilde{M}$ has $2n-k$ zeroes. Using Lemmas \ref{lem:indperpind}, \ref{lem:projindperpprojind} and \ref{upstairsdownstairs} we calculate:
 \begin{align*}
 2\chi(M) & = k + \chi(\widetilde{M})  \\
          & = k + \sum_{i=1}^n \sum_{y\in \pi^{-1}(x_i)} \ind_v(y)\\
          & = k + \sum_{i=1}^n \sum_{y\in \pi^{-1}(x_i)} \left(\ind^\perp_v(y)+1\right)\\
          & = k +(2n-k) + \sum_{i=1}^n \sum_{y\in \pi^{-1}(x_i)} \ind^\perp_v(y)\\
          & = 2n + \sum_{i=1}^n \pind^\perp_\xi(x_i) \\
          & = 2n +\sum_{i=1}^n \left(\pind_\xi(x_i)-2\right) \\
          & = \sum_{i=1}^n \pind_\xi(x_i).
 \end{align*}

In the case $m$ odd, the statement of Theorem \ref{maintheorem} becomes
\[
\sum_{i=1}^n \pind_\xi(x_i) \equiv_2 0,
\]
which is a trivial consequence of \lemref{lem:oddpindzero} above. \qed

\begin{rem}
Theorem \ref{maintheorem} may seem
to hold little
content in the odd-dimensional case (besides the content of Lemma \ref{lem:oddpindzero}). We included its statement here because, in light of Example \ref{RP3}, we believe that it is the best that one can hope for.
\end{rem}

\end{document}